\documentclass{amsart}

\usepackage{bbm}
\usepackage{hyperref}

\newcommand*{\mailto}[1]{\href{mailto:#1}{\nolinkurl{#1}}}

\newtheorem{theorem}{Theorem}[section]

\newtheorem{lemma}[theorem]{Lemma}
\newtheorem{proposition}[theorem]{Proposition}
\newtheorem{corollary}[theorem]{Corollary}

\newcommand{\R}{{\mathbb R}}
\newcommand{\N}{{\mathbb N}}

\newcommand{\C}{{\mathbb C}}

\newcommand{\spr}[2]{\langle #1 , #2 \rangle}

\newcommand{\indik}{\mathbbm{1}}
\newcommand{\E}{\mathrm{e}}

\newcommand{\im}{\mathrm{Im}}

\newcommand{\supp}{\mathrm{supp}}

\newcommand{\M}{\mathcal{M}}
\newcommand{\SP}{\mathcal{S}}
\newcommand{\T}{\mathcal{T}}
\newcommand{\Lrom}{L^2((a,b);\omega)}

\renewcommand{\labelenumi}{(\roman{enumi})}
\renewcommand{\theenumi}{\roman{enumi}}
\numberwithin{equation}{section}

\begin{document}

\title[Two inverse spectral problems for singular Krein strings]{Two inverse spectral problems for~a~class~of~singular~Krein~strings}

\author[J.\ Eckhardt]{Jonathan Eckhardt}
\address{Faculty of Mathematics\\ University of Vienna\\
Nordbergstrasse 15\\ 1090 Wien\\ Austria}
\email{\mailto{jonathan.eckhardt@univie.ac.at}}
\urladdr{\url{http://homepage.univie.ac.at/jonathan.eckhardt/}}

\thanks{\href{http://dx.doi.org/10.1093/imrn/rnt049}{Int.\ Math.\ Res.\ Not.\ IMRN {\bf 2014}, no.~13, 3692--3713}}
\thanks{{\it Research supported by the Austrian Science Fund (FWF) under Grant No.\ Y330}}


\begin{abstract}
 We solve the inverse problem from the spectral measure and the inverse three-spectra problem for the class of singular Krein strings on a finite interval with trace class resolvents. 
 In particular, this includes a complete description of all possible spectral measures and three (Dirichlet) spectra associated with this class of Krein strings.  
 The solutions of these inverse problems are obtained by approximation with Stieltjes strings.  
\end{abstract}

\maketitle

\section{Introduction}

 Let $(a,b)$ be some bounded interval and consider the class $\M$ of all (positive) Borel measures $\omega$ on $(a,b)$ for which the integral 
 \begin{align*}
  \int_{a}^{b} (b-x)(x-a) d\omega(x) 
 \end{align*}
 is finite. The spectral problem for a string with fixed endpoints and mass distribution given by some $\omega\in\M$ is the boundary value problem  
 \begin{align*}
  -u'' = z\, u\, \omega  
 \end{align*}
 on $(a,b)$ with Dirichlet boundary conditions at the endpoints and a complex spectral parameter $z\in\C$. 
 Of course, this differential equation has to be understood in a distributional sense (we refer to Section~\ref{sec2} for some details).
 The aim of the present article is to solve two inverse spectral problems for this class of singular strings. 
 
 First of all, for each such string with mass distribution $\omega\in\M$ there is an associated scalar spectral measure (cf.\ \cite[\S 6]{ka67}, \cite[Subsection~2.7]{kac}) which belongs to the class $\SP$ of all (positive) discrete measures $\rho$ on $\R^+$ for which the sum
 \begin{align*}
  \sum_{\lambda\in\supp(\rho)} \frac{1}{\lambda}
 \end{align*}
 is finite. In Section~\ref{secISP} we will solve the corresponding inverse problem, that is, we show that the correspondence between mass distributions in $\M$  and spectral  measures in $\SP$ is one-to-one. 
 Moreover, we are also able to characterize those mass distributions which are finite near some endpoint in terms of the asymptotic behavior of the corresponding spectral measure. 
 Of course, the case of a regular left endpoint overlaps with Krein's inverse spectral theory of strings on a half-axis \cite{dymmck}, \cite{kackrein}, \cite{kac}, \cite{kotwat} (up to the minor differences regarding heavy endpoints and boundary conditions). 
 At this point, let us also mention that generalizations to strings with singular left endpoints have already been considered in \cite{ka67} and some inverse spectral results have been obtained in \cite{ko75}, \cite{ko76}, \cite{ko07}. 
 In particular, \cite{ko75} considers the class of singular strings with the left endpoint being $-\infty$ and the mass distribution satisfying a growth restriction there, which guarantees that the resolvents of the strings, restricted to a neighborhood of $-\infty$, are trace class (cf.\ \cite[Subsection~2.7]{kac}).
 There, it has been shown that a corresponding spectral measure determines the mass distribution of such a singular string up to a possible shift. 
 The proof of this result is very similar to ours, which also relies on a variant of de Branges' subspace ordering theorem \cite{ko76}. 
 Regarding the existence part of this inverse problem, it has been shown in \cite[Theorem~4.1]{ko75} that at least measures with a polynomial growth restriction arise as spectral measures of such singular strings. 
 This result has been enhanced in \cite{ko07}, where the class of singular strings corresponding to this polynomial growth restriction of the spectral measure has been described explicitly in terms of a growth restriction of the mass distribution near $-\infty$. 
 The present article shows that it is also possible to describe the set of spectral measures corresponding to the class of mass distributions with an additional growth restriction near the right endpoint (which ensures the resolvents to be trace class).  
 In contrast to the results in \cite{ko75}, \cite{ko07}, spectral measures from this class may indeed exhibit arbitrary growth. 
  
 Secondly, we will consider the inverse three-spectra problem for this class of strings. 
  By the three (Dirichlet) spectra associated with a mass distribution $\omega\in\M$ (and a fixed interior point $c\in(a,b)$) we mean the spectra $\sigma(S)$, $\sigma(S_a)$, and $\sigma(S_b)$, where $\sigma(S)$ is the spectrum of the whole string and $\sigma(S_a)$, $\sigma(S_b)$ are the spectra of the string restricted to $(a,c)$, $(c,b)$, respectively with an additional Dirichlet boundary condition at the point $c$. 
  The corresponding inverse problem has been treated for example in \cite{gessimts}, \cite{hrymyk}, \cite{pivovarchik} for Schr\"odinger operators (corresponding to smooth strings) and in \cite{boypiv} for Stieltjes strings (that is, strings consisting of a finite number of point masses). 
  It turns out that these three spectra belong to the class $\T$ which consists of all triples $(\sigma, \sigma_a, \sigma_b)$ of discrete subsets of $\R^+$ for which the sum 
  \begin{align*}
   \sum_{\lambda\in\sigma} \frac{1}{\lambda}
  \end{align*}
  is finite,  the intersection $\sigma_a\cap\sigma_b$ is contained in $\sigma$ and the set $\sigma_a\cup\sigma_b$ interlaces $\sigma\backslash(\sigma_a\cap\sigma_b)$. 
  To be precise, we say some discrete set $A$ interlaces $B$ if $b_1 < a_1 < b_2 < a_2 < \cdots $, where $A=\lbrace a_k\rbrace_{k=1}^{n_a}$ and $B=\lbrace b_k \rbrace_{k=1}^{n_b}$ for some $n_a$, $n_b\in\N_0\cup\lbrace\infty\rbrace$ and strictly increasing sequences $(a_k)_{k=1}^{n_a}$ and $(b_k)_{k=1}^{n_b}$. 
   Moreover, if one of these sets is finite, then so has to be the other one with $n_b=n_a$ or $n_b=n_a+1$ and they either end with $\cdots < b_{n_b} < a_{n_a}$ or with $\cdots < a_{n_a} < b_{n_b}$. 
  Alternatively, we could also define $\T$ to consist of all triples $(\sigma,\sigma_a,\sigma_b)$ of discrete subsets of $\R^+$  such that for each $\lambda\in\sigma$ we have $\lambda\in\sigma_a$ if and only if $\lambda\in\sigma_b$ and for which the function
 \begin{align*}
   \prod_{\lambda\in\sigma} \biggl(1-\frac{z}{\lambda}\biggr)^{-1} \prod_{\mu_a\in\sigma_a} \biggl(1-\frac{z}{\mu_a}\biggr) \prod_{\mu_b\in\sigma_b} \biggl(1-\frac{z}{\mu_b}\biggr), \quad z\in\C\backslash\R, 
 \end{align*}
  (where the products are assumed to converge locally uniformly) is a Herglotz--Nevanlinna function, i.e., maps the upper complex half-plane into itself or is a real constant.
 The equivalence of these definitions follows essentially from the interlacing condition for (necessarily simple) zeros and poles of meromorphic Herglotz--Nevanlinna functions; see e.g.\ \cite[Theorem~2.1]{gessimts}. 
 In contrast to the inverse problem from the spectral measure, there is no one-to-one correspondence between $\M$ and $\T$, as already noticed in \cite{gessimts} for Schr\"odinger operators and in \cite{boypiv} for Stieltjes strings.
   Nevertheless, the map $\M\rightarrow\T$ turns out to be onto and we are able to describe all possible strings with the same three spectra in terms of additional spectral data. 
   
 The proofs of our theorems rely on the corresponding results for Stieltjes strings, discussed in \cite{bss}, \cite{boypiv} or less immediately applicable (due to the difference regarding heavy endpoints and boundary conditions) also in \cite{dymmck}, \cite{kackrein}, \cite{kotwat}.
  In Section~\ref{secCont} we will prove some continuity results for the spectral quantities on isospectral sets of $\M$, which allow us to lift the results for Stieltjes strings to our class of mass distributions. The two inverse problems will be discussed in the consecutive sections. 
 
  Regarding integration of a continuous function $g$ on $(a,b)$ with respect to some measure $\omega\in\M$ we will use the following convenient notation 
\begin{align*}
 \int_\alpha^\beta g(x)d\omega(x) = \begin{cases}
                                     \int_{[\alpha,\beta)} g(x) d\omega(x),          &  \alpha < \beta, \\
                                     0,                                     &  \alpha=\beta, \\
                                     -\int_{[\beta,\alpha)} g(x) d\omega(x),         &  \alpha >\beta, 
                                    \end{cases}
\end{align*} 
 for $\alpha$, $\beta\in(a,b)$. In particular, the integration by parts formula takes the form 
\begin{align*}
 \int_\alpha^\beta f(x) g(x) d\omega(x) =  f(\beta) \int_\alpha^\beta g(x)d\omega(x) - \int_\alpha^\beta f'(x) \int_\alpha^x g(s)d\omega(s)\, dx
\end{align*}
 for locally absolutely continuous functions $f$ on $(a,b)$, which will be used frequently.

\section{Spectral theory for a singular string}\label{sec2}

 In this section we will discuss the spectral problem for a string with fixed endpoints and arbitrary mass distribution $\omega\in\M$, as far as it is needed to solve our inverse problems. 
 A solution $u$ of the differential equation  
 \begin{align}\label{eqnString}
  -u'' = z\, u\, \omega 
 \end{align}
 on the interval $(a,b)$ with a complex spectral parameter $z\in\C$, is a complex-valued, locally absolutely continuous function on $(a,b)$ such that 
 \begin{align}\label{eqnDEM}
  u'(\alpha) - u'(\beta) = z \int_\alpha^\beta u(x) d\omega(x) 
 \end{align}
 for almost all $\alpha$, $\beta\in(a,b)$.  
 In particular, the derivative of a solution $u$ of~\eqref{eqnString} has a unique left-continuous representative which is locally of bounded variation.
 Here and henceforth, for definiteness, this representative will always be denoted with $u'$ such that~\eqref{eqnDEM} holds for all $\alpha$, $\beta\in(a,b)$.
 Together with the boundary conditions
 \begin{align}\label{eqnBC}
      \lim_{\alpha\rightarrow a} u(\alpha)-u'(\alpha)(\alpha-a)= 0 \quad\text{and}\quad \lim_{\beta\rightarrow b} u(\beta) - u'(\beta) (\beta-b) = 0
 \end{align}
 (these limits are known to exist for solutions of \eqref{eqnString} which are square integrable with respect to $\omega$ near the respective endpoint), equation~\eqref{eqnString} gives rise to a unique self-adjoint linear operator $S$ in the weighted Hilbert space $\Lrom$ (see e.g.\ \cite[Section~6]{measureSL} for details).
  Hereby note that our differential equation is in the limit-circle case at an endpoint if and only if $\omega$ is finite near this endpoint. In this case the boundary conditions reduce to Dirichlet boundary conditions. Otherwise, in the limit-point case the boundary conditions are known to be superfluous.  
 As a first step we introduce solutions of~\eqref{eqnString}, which are square integrable with respect to $\omega$ near one endpoint, satisfy the boundary condition there and depend analytically on the spectral parameter (cf.\ \cite[Theorem~9]{ka67}).
 
 \begin{theorem}\label{thmPHI}
 For each $z\in\C$ there is a unique solution $\phi_a(z,\cdot\,)$ of~\eqref{eqnString} with the spatial asymptotics 
 \begin{align*}
 \phi_a(z,x) \sim x-a \quad\text{and}\quad \phi_a'(z,x) \sim 1
 \end{align*} 
 as $x\rightarrow a$. Moreover, the functions $\phi_a(\,\cdot\,,x)$ and $\phi_a'(\,\cdot\,,x)$ are real entire and of finite exponential type for each $x\in(a,b)$.
 \end{theorem}

\begin{proof}
 First of all we show that for each $z\in\C$, the integral equation
\begin{align}\label{eqnIsoInt}
 m_a(z,x) = 1 - z \int_{a}^{x}  \frac{(x-s)(s-a)}{x-a} m_a(z,s)d\omega(s), \quad x\in(a,b),
\end{align}
has a unique bounded continuous solution $m_a(z,\cdot\,)$.
Therefore consider the integral operator $K_a$ on $C_b(a,b)$ defined by 
\begin{align*}
 K_a f(x) = \int_{a}^{x}  \frac{(x-s)(s-a)}{x-a} f(s)d\omega(s), \quad x\in(a,b),~f\in C_b(a,b),
\end{align*}
where $C_b(a,b)$ is the space of bounded continuous functions on $(a,b)$. 
 We show that for each $f\in C_b(a,b)$ and $n\in\N$ we have the estimate
\begin{align*}
 \sup_{a<s<x} \left|K_a^n f(s)\right| \leq \frac{1}{n!} \biggl(\int_{a}^x p(s) d\omega(s)\biggr)^n \sup_{a<s<x} \left|f(s)\right|, \quad x\in(a,b),
\end{align*}
where $p$ is the polynomial given by $(b-a)p(s)=(b-s)(s-a)$, $s\in(a,b)$. 
 In fact, from a simple estimate we get inductively
\begin{align*}
 \sup_{a<s<x} \left|K_a^n f(s)\right| & \leq \sup_{a<s<x} \int_{a}^s  \left|K_a^{n-1}f(r)\right| p(r) d\omega(r) \\
  & \leq \frac{1}{(n-1)!} \int_{a}^x  \biggl(\int_{a}^r p d\omega\biggr)^{n-1} p(r) d\omega(r)\, \sup_{a<s<x} \left|f(s)\right|, \quad x\in(a,b).
\end{align*}
Now an application of the substitution rule for Lebesgue--Stieltjes integrals~\cite{tsub} yields the claim. In particular this shows  that 
\begin{align}\label{eqnKneu}
 \|K_a^n\| \leq \frac{1}{n!} \biggl(\int_a^b \frac{(b-x)(x-a)}{b-a} d\omega(x) \biggr)^n, \quad n\in\N,
\end{align}
and hence the Neumann series 
\begin{align*}
 m_a(z,x) = \sum_{n=0}^\infty (-1)^n z^n K_a^n 1(x) = (I+zK_a)^{-1} 1(x), \quad x\in(a,b),~z\in\C,
\end{align*}
converges absolutely, uniformly in $x\in(a,b)$ and even locally uniformly in $z\in\C$.
 Thus $m_a(z,\cdot\,)$ is the unique solution in $C_b(a,b)$ of the integral equation~\eqref{eqnIsoInt}. Moreover, integrating the right-hand side of~\eqref{eqnIsoInt} by parts shows that this function is locally absolutely continuous with derivative given by
\begin{align}\label{eqnIntEqnDer}
 m_a'(z,x) =  - z\int_{a}^{x} \biggl(\frac{s-a}{x-a}\biggr)^2 m_a(z,s)d\omega(s), \quad x\in(a,b),~z\in\C.
\end{align}
Therefore, for each $z\in\C$ we have the spatial asymptotics 
\begin{align*}
 m_a(z,x) \rightarrow 1 \quad\text{and}\quad m_a'(z,x)\rightarrow 0,
\end{align*}
as $x\rightarrow a$. Indeed, this follows from the integral equation~\eqref{eqnIsoInt} and~\eqref{eqnIntEqnDer} in view of the fact that the function $m_a(z,\cdot\,)$ is uniformly bounded. 
 Now the functions 
\begin{align*}
\phi_a(z,x) = (x-a) m_a(z,x), \quad x\in(a,b),~z\in\C,
\end{align*}
 satisfy the integral equations 
\begin{align*}
 \phi_a(z,x) & = x-a - z \int_{a}^x (x-s) \phi_a(z,s)d\omega(s),\quad x\in(a,b),~z\in\C,
\end{align*}
 and hence are solutions of~\eqref{eqnString} (see e.g.~\cite[Proposition~3.3]{measureSL}).
The spatial asymptotics of $\phi_a(z,\cdot\,)$ near $a$ easily follow from the corresponding results for the function $m_a(z,\cdot\,)$. 
 Also note that these asymptotics uniquely determine the solution $\phi_a(z,\cdot\,)$. 
Finally, the Neumann series and the estimates in~\eqref{eqnKneu} guarantee that $m_a(\,\cdot\,,x)$ is real entire and of finite exponential type, uniformly for all $x\in(a,b)$. Hence we see from~\eqref{eqnIntEqnDer} that $m_a'(\,\cdot\,,x)$ is also real entire with finite exponential type for each $x\in(a,b)$. Of course, this proves that the functions $\phi_a(\,\cdot\,,x)$ and $\phi_a'(\,\cdot\,,x)$ are real entire and of finite exponential type for each $x\in(a,b)$. 
\end{proof}

 Of course, a similar calculation shows that for each $z\in\C$ there is a unique solution $\phi_b(z,\cdot\,)$ of~\eqref{eqnString} with the spatial asymptotics
 \begin{align*}
  \phi_b(z,x) \sim b-x \quad\text{and}\quad \phi_b'(z,x) \sim -1
 \end{align*}
 as $x\rightarrow b$. Again the functions $\phi_b(\,\cdot\,,x)$ and $\phi_b'(\,\cdot\,,x)$ are real entire and of finite exponential type for each $x\in(a,b)$. 
 Note that, because of  the spatial asymptotics, for $z=0$ these solutions are given explicitly by
 \begin{align*}
  \phi_a(0,x) = x-a \quad\text{and}\quad \phi_b(0,x) = b-x, \quad x\in(a,b).
 \end{align*}
 Furthermore, the solutions $\phi_a(z,\cdot\,)$ and $\phi_b(z,\cdot\,)$, $z\in\C$ are square integrable with respect to $\omega$ near $a$, $b$ respectively and satisfy the boundary condition~\eqref{eqnBC} there. 
 In particular, this guarantees that the spectrum of $S$ is purely discrete and simple in view of~\cite[Theorem~8.5]{measureSL} and~\cite[Theorem~9.6]{measureSL}.   Consequently, some $\lambda\in\C$ is an eigenvalue of $S$ if and only if the solutions $\phi_a(\lambda,\cdot\,)$ and $\phi_b(\lambda,\cdot\,)$ are linearly dependent, that is, their Wronskian 
\begin{align*}
 W(z) = \phi_b(z,x)\phi_a'(z,x) - \phi_b'(z,x)\phi_a(z,x), \quad x\in(a,b),~z\in\C,
\end{align*}
vanishes in $\lambda$. 
 In this case we have   
\begin{align*} 
 \phi_b(\lambda,x) = (-1)^{\vartheta_\lambda} c_{\lambda} \phi_a(\lambda,x), \quad x\in(a,b),
\end{align*}
 for some $\vartheta_\lambda\in\lbrace0, 1\rbrace$ and some positive real $c_{\lambda}\in\R^+$, referred to as the coupling constant associated with the eigenvalue $\lambda$.  Moreover, the quantity 
\begin{align*}
 \gamma_{\lambda}^{2} = \int_a^b |\phi_a(\lambda,x)|^2 d\omega(x)
\end{align*}
 is finite and referred to as the norming constant associated with the eigenvalue $\lambda$.
 An integration by parts, using the spatial asymptotics of our solutions shows that 
 \begin{align*}
  \lambda\gamma_{\lambda}^{2} = \lambda \int_a^b |\phi_a(\lambda,x)|^2 d\omega(x) = \int_a^b |\phi_a'(\lambda,x)|^2 dx,
 \end{align*}
 which guarantees that the spectrum of $S$ is strictly positive. The following lemma relates all these spectral quantities.

\begin{lemma}\label{lemWderlam}
 For each eigenvalue $\lambda\in\sigma(S)$ we have
\begin{align}\label{eqnWlam}
 - \dot{W}(\lambda) = \int_a^b \phi_a(\lambda,x)\phi_b(\lambda,x)d\omega(x) = (-1)^{\vartheta_\lambda} c_{\lambda} \gamma_{\lambda}^{2} \not=0,
\end{align}
where the dot denotes differentiation with respect to the spectral parameter. 
\end{lemma}

\begin{proof}
We set
\begin{align}\label{eqnWpm}
 W_a(z,x) = \dot{\phi}_a(z,x) \phi_b'(z,x) - \dot{\phi}'_a(z,x)\phi_b(z,x), \quad x\in(a,b),~z\in\C,  
\end{align}
where the spatial differentiation is done first. Now, since $\phi_a(z,\cdot\,)$ and $\phi_b(z,\cdot\,)$ satisfy~\eqref{eqnString} one gets
\begin{align*}
 W_a(z,\beta) - W_a(z,\alpha) = \int_\alpha^\beta \phi_a(z,s) \phi_b(z,s)d\omega(s), \quad \alpha,\,\beta\in(a,b).
\end{align*}
More precisely, this follows by differentiating~\eqref{eqnWpm} with respect to the spatial variable, where the derivative is in general a Borel measure.
Now differentiating the integral equation~\eqref{eqnIsoInt} and~\eqref{eqnIntEqnDer} with respect to the spectral variable we get
\begin{align*}
 (I+zK_a) \dot{m}_a(z,\cdot\,) = - K_a m_a(z,\cdot\,), \quad z\in\C,
\end{align*}
as well as
\begin{align*}
 \dot{m}_a'(z,x) =  -\int_a^{x} \biggl(\frac{s-a}{x-a}\biggr)^2  \left(m_a(z,s)+z\dot{m}_a(z,s)\right) d\omega(s), \quad x\in(a,b),~z\in\C.
\end{align*}
 In particular this shows that $\dot{m}_a(z,x) \rightarrow 0$ and $\dot{m}_a'(z,x) \rightarrow 0$ as $x\rightarrow a$ for each $z\in\C$. If $\lambda\in\sigma(S)$ is an eigenvalue, then we furthermore know
\begin{align*}
 (-1)^{\vartheta_\lambda} c_\lambda (x-a) m_a(\lambda,x) = (-1)^{\vartheta_\lambda} c_\lambda \phi_a(\lambda,x) =  \phi_b(\lambda,x), \quad x\in(a,b).
\end{align*}
 Hence after a simple calculation we see that   
\begin{align*}
 W_a(\lambda,x) = (-1)^{\vartheta_\lambda} c_\lambda (x-a)^2 \left( \dot{m}_a(\lambda,x) m_a'(\lambda,x) - \dot{m}_a'(\lambda,x) m_a(\lambda,x)\right),
\end{align*}
  tends to zero as $x\rightarrow a$, which shows  
\begin{align*}
 W_a(\lambda,x) = \int_a^{x} \phi_a(\lambda,s) \phi_b(\lambda,s) d\omega(s), \quad x\in(a,b).
\end{align*}
 Finally, since a similar equality holds for the function
\begin{align*}
 W_b(z,x) = \dot{\phi}_b(z,x) \phi_a'(z,x) - \dot{\phi}'_b(z,x)\phi_a(z,x), \quad x\in(a,b),~z\in\C, 
\end{align*}
 we end up with
\begin{align*}
 - \dot{W}(\lambda) = W_a(\lambda,x) - W_b(\lambda,x) =  \int_a^b \phi_a(\lambda,s)\phi_b(\lambda,s)d\omega(s), \quad x\in(a,b),
\end{align*}
which is the claimed identity.
\end{proof}

Since the eigenfunctions $\phi_a(\lambda,\cdot\,)$, $\lambda\in\sigma(S)$ form a complete orthonormal system, the transformation
\begin{align*}
 \mathcal{F}f(\lambda) = \int_a^b f(x)\phi_a(\lambda,x)d\omega(x), \quad \lambda\in\sigma(S),~f\in\Lrom, 
\end{align*}
 is unitary from $\Lrom$ onto $L^2(\R;\rho)$, where the discrete measure 
 \begin{align*}
  \rho = \sum_{\lambda\in\sigma(S)} \gamma_\lambda^{-2} \delta_\lambda
 \end{align*}
 is referred to as the spectral measure associated with $S$. Here $\delta_\lambda$ is the Dirac measure in the point $\lambda\in\sigma(S)$. 
 It is not hard to show that this transformation maps $S$ onto multiplication with the independent variable in $L^2(\R;\rho)$.  
 In particular, if $z\not\in\sigma(S)$ we have
 \begin{align*}
  \int_a^b (S-z)^{-1}f(x) g(x)^\ast d\omega(x) = \int_\R \frac{\mathcal{F}f(\lambda) \mathcal{F}g(\lambda)^\ast}{\lambda-z} d\rho(\lambda)
 \end{align*}
 for all $f$, $g\in\Lrom$. Hereby note that the resolvent is given by
\begin{align*}
 (S-z)^{-1} f(x) = \int_a^b G(z,x,s) f(s) d\omega(s), \quad x\in(a,b),~ f\in\Lrom,
\end{align*}
where $G$ is the Green function
\begin{align*}
 G(z,x,y) = W(z)^{-1} \begin{cases}
             \phi_a(z,x) \phi_b(z,y), & y\in [x,b), \\
             \phi_a(z,y) \phi_b(z,x), & y\in (a,x).
            \end{cases}
\end{align*}
 In fact, this follows from~\cite[Theorem~8.3]{measureSL} since $\phi_a(z,\cdot\,)$ satisfies the boundary condition near $a$ and $\phi_b(z,\cdot\,)$ the one near $b$. The following proposition shows that our spectral measure actually lies in the class $\SP$. 
 
\begin{proposition}\label{propInverse}
The inverse of $S$ is a trace class operator with
\begin{align} \label{eqnTrace}
  \sum_{\lambda\in\sigma(S)} \frac{1}{\lambda} = \int_a^b  \frac{(b-x)(x-a)}{b-a} d\omega(x).
\end{align}
\end{proposition}

\begin{proof}
 Since the solutions $\phi_a(0,\cdot\,)$ and $\phi_b(0,\cdot\,)$ are linearly independent, $S$ is invertible with
 \begin{align*}
  S^{-1} f(x) = \int_a^b \frac{(b-\max(x,s))(\min(x,s)-a)}{b-a} f(s) d\omega(s), \quad x\in(a,b),
 \end{align*}
 for every $f\in\Lrom$. Moreover, since the spectrum of $S$ is positive, we infer from Mercer's theorem (see e.g.\ \cite[\S10.1]{gokr} or \cite[Theorem~3.a.1]{koe}) that $S^{-1}$ is a trace class operator with trace given as in the claim. 
\end{proof}

Next consider the self-adjoint operator $S_a$ in $L^2((a,c);\omega)$ associated with~\eqref{eqnString} and Dirichlet boundary conditions at $a$ (if necessary) and $c$. For similar reasons as above, the spectrum of this operator is strictly positive and consists of all zeros of the entire function $\phi_a(\,\cdot\,,c)$. 
 Furthermore, we introduce the self-adjoint operator $S_a'$ in $L^2((a,c);\omega)$ associated with~\eqref{eqnString}, Dirichlet boundary conditions at $a$ (if necessary) and Neumann boundary conditions at $c$. 
 Again, the spectrum of $S_a'$ is positive and consists precisely of the zeros of the entire function $\phi_a'(\,\cdot\,,c)$. 

\begin{theorem}\label{thmPHIrep}
 The entire functions $\phi_a(\,\cdot\,,c)$ and $\phi_a'(\,\cdot\,,c)$ are of exponential type zero and given by 
\begin{align*}
 \phi_a(z,c) = (c-a) \prod_{\mu_a\in \sigma(S_a)} \biggl(1-\frac{z}{\mu_a}\biggr) \quad\text{and}\quad  \phi_a'(z,c) = \prod_{\nu_a\in \sigma(S_a')} \biggl(1-\frac{z}{\nu_a}\biggr) 
\end{align*}
for each $z\in\C$.
\end{theorem}

\begin{proof}
The inverse of $S_a'$ is given by 
\begin{align*}
 S_a'^{-1} f(x) = \int_a^c (\min(x,s)-a)f(s)d\omega(s), \quad x\in(a,c),~f\in L^2((a,c);\omega).
\end{align*}
As in the proof of Proposition~\ref{propInverse} one infers that this inverse is trace class with 
\begin{align*}
 \sum_{\nu_a\in \sigma(S_a')} \frac{1}{\nu_a} = \int_{a}^c (s-a) d\omega(s).
\end{align*}
 Moreover, since the real entire function $\phi_a'(\,\cdot\,,c)$ is of finite exponential type with summable zeros, the Hadamard factorization shows that 
\begin{align}\label{eqnPHIreppre}
 \phi_a'(z,c) = \E^{B_a' z} \prod_{\nu_a\in\sigma(S_a')} \biggl( 1-\frac{z}{\nu_a}\biggr), \quad z\in\C,
\end{align}
for some $B_a'\in\R$.
 From the integral equations for the function $m_a$ and its derivative as well as the representation~\eqref{eqnPHIreppre} on the other side we get
\begin{align*}
- \int_a^c (s-a)d\omega(s) = \dot{\phi}_a'(0,c) = B_a' - \sum_{\nu_a\in\sigma(S_a')}\frac{1}{\nu_a},
\end{align*}
 which yields $B_a'=0$.  Hence $\phi_a'(\,\cdot\,,c)$ has the claimed representation and therefore is of exponential type zero. 
 In much the same manner one shows the claimed properties for $\phi_a(\,\cdot\,,c)$. 
\end{proof}

 Of course, we may as well introduce the self-adjoint operator $S_b$ in $L^2([c,b);\omega)$ associated with~\eqref{eqnString} and Dirichlet boundary conditions at $c$ and $b$ (if necessary). 
 To be more precise, note that this operator is actually multi-valued if and only if $\omega$ has mass in the point $c$ (see \cite[Corollary~7.4]{measureSL} for details). 
 Nevertheless, the spectrum of $S_b$ is positive and consists precisely of all zeros of the entire function $\phi_b(\,\cdot\,,c)$. 
 Similarly, we denote with $S_b'$ the self-adjoint operator in $L^2([c,b);\omega)$ associated with~\eqref{eqnString}, Neumann boundary conditions at $c$ and Dirichlet boundary conditions at $b$ (if necessary). Again, the spectrum of this operator is positive and consists of all zeros of $\phi_b'(\,\cdot\,,c)$.
 Now essentially by reflection we infer that the entire functions $\phi_b(\,\cdot\,,c)$ and $\phi_b'(\,\cdot\,,c)$ are of exponential type zero as well and given by 
\begin{align*}
 \phi_b(z,c) = (b-c)\prod_{\mu_b\in \sigma(S_b)} \biggl( 1-\frac{z}{\mu_b}\biggr) \quad\text{and}\quad \phi_b'(z,c) = -\prod_{\nu_b\in\sigma(S_b')} \biggl( 1-\frac{z}{\nu_b}\biggr)
\end{align*}
 for each $z\in\C$. 
 As a consequence, we furthermore obtain the product representation
 \begin{align*}
 W(z) = (b-a) \prod_{\lambda\in\sigma(S)} \biggl(1-\frac{z}{\lambda}\biggr), \quad z\in\C,
\end{align*}
 since $W$ is of exponential type zero with summable zeros and $W(0)=b-a$. 
 
 The spectra $\sigma(S)$, $\sigma(S_a)$, $\sigma(S_b)$ are referred to as the three (Dirichlet) spectra associated with $\omega$ (and the interior point $c$). 
 Since for each eigenvalue $\lambda$ of $S$ the solutions $\phi_a(\lambda,\cdot\,)$ and $\phi_b(\lambda,\cdot\,)$ are linearly dependent, this eigenvalue lies in $\sigma(S_a)$ if and only if it lies in $\sigma(S_b)$.   
 Moreover, the following proposition shows that the triple $(\sigma(S), \sigma(S_a), \sigma(S_b))$ of discrete sets actually belongs to $\T$.
 
 \begin{proposition}
  $G(\,\cdot\,,c,c)$ is a meromorphic Herglotz--Nevanlinna function.
 \end{proposition}
 
 \begin{proof}
 A calculation using the Lagrange identity shows that 
 \begin{align*}
  \im\, G(z,c,c) = \im\, z \int_a^b |G(z,c,s)|^2 d\omega(s), \quad z\in\C\backslash\R,
 \end{align*}
  which proves the claim. 
\end{proof}
 
 Finally, note that the norming constant $\gamma_\lambda^2$ for some eigenvalue $\lambda\in\sigma(S)$ is uniquely determined by the three spectra unless $\phi_a(\lambda,c)=\phi_b(\lambda,c) = 0$. In fact, in this case it can be written down explicitly in terms of these spectra in view of 
\begin{align}\label{eqngamma}
 \gamma_{\lambda}^2 =  - \dot{W}(\lambda) \phi_a(\lambda,c) \phi_b(\lambda,c)^{-1}. 
\end{align}
 Hence one sees that the spectral measure may be recovered from the three spectra 
 and the collection of coupling constants $c_\lambda$, $\lambda\in\sigma(S)\cap\sigma(S_a)\cap\sigma(S_b)$.

\section{Continuity of the spectral quantities}\label{secCont}

For the solution of our inverse problems, we need some continuity of the spectral quantities. 
 Therefore we equip $\M$ with the initial topology with respect to the linear functionals
 \begin{align*}
  \omega \mapsto \int_a^b f(x) (b-x)(x-a)d\omega(x), \quad f\in C_0(a,b),
 \end{align*}
 on $\M$, where $C_0(a,b)$ is the set of all continuous functions on $(a,b)$ which vanish in $a$ and $b$. 
 Note that this is the weak$^\ast$ topology upon identifying $\M$ with a (weak$^\ast$ closed) subset of the dual of $C_0(a,b)$. 
 In fact, each measure $\omega\in\M$ may be regarded as the bounded linear functional 
 \begin{align*}
  f \mapsto \int_a^b f(x)(b-x)(x-a)d\omega(x)
 \end{align*}
 on $C_0(a,b)$. 
 As a consequence, our topology is metrizable on bounded (with respect to the operator norm of the corresponding functionals) subsets of $\M$ and convergent sequences in $\M$ are bounded. 
 In order to investigate continuous dependence of the spectral quantities, consider some sequence $\omega_n\in\M$, $n\in\N$. 
 We write $\omega_n\rightharpoonup^\ast\omega$ if this sequence converges to $\omega$ as $n\rightarrow\infty$ with respect to the weak$^\ast$ topology.
 All quantities corresponding to the measures $\omega$ and $\omega_n$, $n\in\N$ are denoted as in the preceding section but with an additional subscript $n\in\N$ for those of $\omega_n$. 
  
  \begin{lemma}\label{lemStrongConv}
   If $\omega_n\rightharpoonup^{\ast}\omega$, then $G_n(\,\cdot\,,c,c)\rightarrow G(\,\cdot\,,c,c)$ locally uniformly on $\C\backslash\R$. 
  \end{lemma}
  
  \begin{proof}
 Consider the Sobolev space $H_0^1(a,b)$ equipped with the definite inner product
 \begin{align*}
  \spr{f}{g}_{H_0^1(a,b)} = \int_a^b f'(x) g'(x)^\ast dx, \quad f,\,g\in H_0^1(a,b),
 \end{align*}
 and note that $H_0^1(a,b)\subseteq \Lrom$ because of the estimate
 \begin{align*}
  |f(x)|^2 \leq \|f\|_{H_0^1(a,b)}^2 \min(b-x,x-a), \quad x\in(a,b),~f\in H_0^1(a,b),
 \end{align*}
 which follows from a simple application of the Cauchy--Schwarz--Bunyakovsky inequality. 
 We introduce the integral operator $R$ on $H_0^1(a,b)$ given by 
 \begin{align*}
  R g(x) = \int_a^b G(0,x,s) g(s)d\omega(s), \quad x\in(a,b),~g\in H_0^1(a,b).
 \end{align*}
 Since $R$ acts like the inverse of $S$, the function $Rg$ is a solution of the differential equation $-Rg'' = g \omega$ on $(a,b)$,
  which is square integrable with respect to $\omega$ and satisfies the boundary conditions in~\eqref{eqnBC}. 
  Moreover, if $g$ has compact support, then this also guarantees $Rg\in H_0^1(a,b)$ (since $Rg'$ is bounded) and an integration by parts shows that 
 \begin{align}\label{eqnLeftDef}
  \spr{f}{Rg}_{H_0^1(a,b)} & = \int_a^b f'(x) Rg'(x)^\ast dx  = \int_a^b f(x) g(x)^\ast d\omega(x) \\
   \nonumber                     & \leq \|f\|_{H_0^1(a,b)} \|g\|_{H_0^1(a,b)} \int_a^b \min(b-x,x-a) d\omega(x)
 \end{align}
 for all $f\in H_0^1(a,b)$. 
 In particular, $R$ is bounded on the dense subspace of functions with compact support.
 Now since for each $x\in(a,b)$, $g\mapsto Rg(x)$ is continuous on $H_0^1(a,b)$, we infer that $R$ is a (well-defined) bounded operator on $H_0^1(a,b)$. Furthermore, by continuity, \eqref{eqnLeftDef} 
 holds for all $g\in H_0^1(a,b)$, which proves that $R$ is even self-adjoint. 
 Furthermore, since $R$ and $S^{-1}$ act the same way one has $\sigma(S^{-1})\backslash\lbrace 0\rbrace=\sigma(R)\backslash\lbrace0\rbrace$.
 In fact, it is readily verified that each non-zero eigenvalue of $S^{-1}$ is also an eigenvalue of $R$. For the converse note that for each function $g\in H_0^1(a,b)$ which is orthogonal to all eigenfunctions $\phi_a(\lambda,\cdot\,)$, $\lambda\in\sigma(S)$ in $H_0^1(a,b)$ we have $g=0$ in $L^2((a,b);\omega)$ in view of~\eqref{eqnLeftDef} and hence $Rg=0$.
 Finally let $\delta_c\in H_0^1(a,b)$ such that $\spr{f}{\delta_c}_{H_0^1(a,b)}=f(c)$ for all $f\in H_0^1(a,b)$. 
 Then, applying variants of the spectral theorem to the operators $R$ and $S$ (in particular, see \cite[Lemma~10.6]{measureSL}) yields for each fixed $z\in\C\backslash\R$ 
 \begin{align*}
  \spr{(R^{-1}-z)^{-1} \delta_c}{\delta_c}_{H_0^1(a,b)} & = \sum_{\lambda\in\sigma(S)} \frac{\phi_a(\lambda,c)}{\lambda-z} \frac{\phi_a(\lambda,c)}{\lambda} \gamma_\lambda^{-2} \\
   & = \spr{G(z,c,\cdot\,)}{G(0,c,\cdot\,)}_{L^2((a,b);\omega)} \\
   & = \frac{G(z,c,c)}{z} -  \frac{1}{z} \frac{(b-c)(c-a)}{b-a},
 \end{align*} 
 where the last equation follows from a direct calculation using the Lagrange identity.  
 
 We will now show that the corresponding operators $R_n$, $n\in\N$ converge to $R$ in the strong operator topology. 
  Therefore, first of all note that given some arbitrary $g\in H_0^1(a,b)$ we have 
 \begin{align*}
  R_n g(x) = \int_a^b G(0,x,s) g(s)d\omega_n(s) \rightarrow \int_a^b G(0,x,s) g(s)d\omega(s) = R g(x), \quad x\in(a,b)
 \end{align*}
 by assumption. Since all these operators are uniformly bounded by~\eqref{eqnTrace}, 
 this implies that $R_n$ converges to $R$ in the weak operator topology. 
 In order to prove that they also converge in the strong operator topology, note that~\eqref{eqnLeftDef} shows 
 \begin{align*}
  \|R_n g\|_{H_0^1(a,b)}^2 = \spr{R_n g}{g}_{L^2((a,b);\omega_n)}, \quad n\in\N.
 \end{align*}
 Moreover, since the functions $R_n g$ are uniformly bounded in $H_0^1(a,b)$, we infer from the Arzel\`{a}--Ascoli theorem that $R_n g$ converges to $R g$ locally uniformly.
 Now if we assume that $g$ has compact support, then we have
 \begin{align*}
  & \left| \int_a^b R_n g(x) g(x)^\ast d\omega_n(x) - \int_a^b Rg(x) g(x)^\ast d\omega(x) \right| \\
  & \qquad\qquad\qquad \leq  \omega_n(\supp(g)) \sup_{x\in\supp(g)} |g(x)| |R_n g(x)-R g(x)| \\
  & \qquad\qquad\qquad\quad + \left|\int_a^b R g(x)g(x)^\ast d\omega_n(x) - \int_a^b Rg(x)g(x)^\ast d\omega(x)\right|,
 \end{align*}
 and thus $R_n g$ converges to $Rg$ in $H_0^1(a,b)$. 
 Again, since our operators are uniformly bounded, we infer that $R_n$ converges to $R$ in the strong operator topology. 
 But now the claim follows since then $(R_n^{-1}-z)^{-1}$ converges to $(R^{-1}-z)^{-1}$ in the strong operator topology (see e.g.\ \cite[Lemma~6.36 and Theorem~6.31]{tschroe}) for each $z\in\C\backslash\R$.  
 \end{proof}
    
 In particular, the preceding lemma implies that each $\lambda\in\sigma(S)$ is the limit of some sequence $\lambda_n\in\sigma(S_n)$, $n\in\N$ if $\omega_n\rightharpoonup^\ast\omega$. 
  Also note that a similar statement holds for the two restricted operators $S_a$ and $S_b$.   
    
 For the sake of simplicity we will restrict our considerations to the set $\M_\sigma$ of all measures in $\omega\in\M$ whose spectrum (associated with the boundary value problem~\eqref{eqnString} and~\eqref{eqnBC}) is contained in a fixed discrete set $\sigma\subset\R^+$ with
 \begin{align*}
  \sum_{\lambda\in\sigma} \frac{1}{\lambda} < \infty.
 \end{align*} 
 Since this set is bounded by Proposition~\ref{propInverse} and closed (as a subset of the dual of $C_0(a,b)$) in view of Lemma~\ref{lemStrongConv}, it is compact by the Banach--Alaoglu theorem.  
  The next result contains some kind of continuity for the three spectra.
   
\begin{proposition}\label{propCont}
  Suppose that the measures $\omega_n$, $n\in\N$ lie in $\M_\sigma$ and $\omega_n\rightharpoonup^{\ast}\omega$. Then there is a subsequence $\omega_{n_k}$ and disjoint sets $\tau_a$, $\tau_b\subseteq\sigma\backslash\sigma(S)$ such that the Wronskians $W_{n_k}$ converge locally uniformly to the function given by 
 \begin{align}\label{eqnWinf}
   W(z) \prod_{\lambda\in\tau_a\cup\tau_b} \biggl(1-\frac{z}{\lambda}\biggr), \quad z\in\C,
 \end{align}
 and the solutions $\phi_{n_k,a}(\,\cdot\,,c)$ and $\phi_{n_k,b}(\,\cdot\,,c)$ converge locally uniformly to the functions  
 \begin{align}\label{eqnPHIinf}
  \phi_a(z,c) \prod_{\lambda\in\tau_a} \biggl(1-\frac{z}{\lambda}\biggr) \quad\text{and}\quad \phi_b(z,c) \prod_{\lambda\in\tau_b} \biggl(1-\frac{z}{\lambda}\biggr), \quad z\in\C,
 \end{align}
 respectively, as $k\rightarrow\infty$. 
 \end{proposition}

 \begin{proof}
  First of all note that for each $n\in\N$ the functions $W_n$, $\phi_{n,a}(\,\cdot\,,c)$ and $\phi_{n,b}(\,\cdot\,,c)$ are bounded by the product 
   \begin{align*}
     (b-a) \prod_{\lambda\in\sigma} \biggl(1+\frac{|z|}{\lambda}\biggr), \quad z\in\C.
   \end{align*} 
  Hence there is a subsequence $\omega_{n_k}$ such that the functions $W_{n_k}$, $\phi_{n_k,a}(\,\cdot\,,c)$ and $\phi_{n_k,b}(\,\cdot\,,c)$ converge locally uniformly to some entire functions of exponential type zero. Moreover, since the zeros of the functions $W_{n_k}$ are contained in $\sigma$, their limit has to be of the form  
  \begin{align*}
   \lim_{k\rightarrow\infty} W_{n_k}(z) = (b-a)\prod_{\lambda\in\sigma_\infty} \biggl(1-\frac{z}{\lambda}\biggr), \quad z\in\C,
  \end{align*}
  for some set $\sigma_\infty\subseteq\sigma$ which contains the spectrum of $S$ in view of Lemma~\ref{lemStrongConv}. 
   Similarly, the functions $\phi_{n_k,a}(\,\cdot\,,c)$ and $\phi_{n_k,b}(\,\cdot\,,c)$ converge to some canonical products which vanish in the points of $\sigma(S_a)$ and $\sigma(S_b)$ respectively, that is, to functions of the form~\eqref{eqnPHIinf}. 
 Now the convergence in Lemma~\ref{lemStrongConv} implies that the sets $\tau_a$ and $\tau_b$ of additional zeros of these limits form a partition of $\sigma_\infty\backslash\sigma(S)$ which proves the claim.  
 \end{proof}

 If we somewhat strengthen the topology on $\M_\sigma$, then we end up with a stronger convergence of the three spectra. In fact, if $\omega_n\rightharpoonup^\ast\omega$ and additionally 
 \begin{align} \label{eqnConvbast}
  \int_a^b (b-x)(x-a)d\omega_n(x) \rightarrow \int_a^b (b-x)(x-a) d\omega(x), 
 \end{align}
 then the spectra $\sigma(S_n)$ actually converge to $\sigma(S)$ (in the sense of pointwise convergence of the respective characteristic functions) and hence $\tau_a$, $\tau_b$ in Proposition~\ref{propCont} are indeed always empty sets.
 Hence the three functions $W_n$, $\phi_{n,a}(\,\cdot\,,c)$ and $\phi_{n,b}(\,\cdot\,,c)$ converge to $W$, $\phi_{a}(\,\cdot\,,c)$ and $\phi_b(\,\cdot\,,c)$, respectively.

 As a simple consequence of Proposition~\ref{propCont} we also obtain some kind of continuity for the norming constants and hence for the spectral measure.

\begin{proposition}\label{propConvSM} 
 Suppose that the measures $\omega_n$, $n\in\N$ lie in $\M_\sigma$ and $\omega_n\rightharpoonup^{\ast}\omega$. Then there is a subsequence $\omega_{n_k}$ and disjoint sets $\tau_a$, $\tau_b\subseteq\sigma\backslash\sigma(S)$ such that the spectra $\sigma(S_{n_k})$ converge to $\sigma(S)\cup\tau_a\cup\tau_b$ and  
  \begin{align}
   \lim_{k\rightarrow\infty} \gamma_{n_k,\lambda}^2 = \begin{cases} 0, & \lambda\in\tau_a, \\ \gamma_{\lambda}^2 \, \prod_{\kappa\in\tau_a} \left(1-\frac{\lambda}{\kappa}\right)^2, & \lambda\in\sigma(S), \\ \infty, & \lambda\in\tau_b.  \end{cases}
  \end{align}
\end{proposition}

 \begin{proof} 
    First of all note that we may assume that for each $n\in\N$, the three spectra $\sigma(S_n)$, $\sigma(S_{n,a})$ and $\sigma(S_{n,b})$ are disjoint. 
     In fact, since this happens only for a countable number of points $c\in(a,b)$, this can be accomplished upon varying $c$. 
     For the same reason we may as well assume that $\sigma(S)$, $\sigma(S_a)$ and $\sigma(S_b)$ are disjoint and hence the claim immediately follows from Proposition~\ref{propCont} and~\eqref{eqngamma}.
 \end{proof}

 Again, if we additionally assume that \eqref{eqnConvbast} holds, then one ends up with stronger convergence of the spectral measures.
  In fact, in this case the norming constants $\gamma_{n,\lambda}^2$ converge to $\gamma_\lambda^2$ for each eigenvalue $\lambda\in\sigma(S)$.

\section{The inverse spectral problem}\label{secISP}

 Using the continuity results from the previous section, we are able to solve the inverse problem from the spectral measure.

\begin{theorem}\label{thmISPee}
 Each $\rho\in\SP$ is the spectral measure of some unique $\omega\in\M$. 
\end{theorem}

 \begin{proof}
  From the inverse problem for Stieltjes strings it is known (see e.g.\ \cite[Theorem~5.5]{bss}) that there are finite measures $\omega_{n}\in\M$, $n\in\N$ with the cut off spectral measures $\indik_{[0,n]}\rho$. 
  Since these measures are uniformly bounded by 
  \begin{align*}
   \int_a^b(b-x)(x-a)d\omega_n(x) \leq (b-a)\sum_{\lambda\in\supp(\rho)} \frac{1}{\lambda}, \quad n\in\N,
  \end{align*}
  we infer from compactness that there is a subsequence $\omega_{n_k}$, $k\in\N$ which converges in the weak$^\ast$ topology to say $\omega\in\M$. 
  Now an application of Proposition~\ref{propConvSM} shows that $\rho$ is the spectral measure associated with $\omega$. 
  
 Uniqueness may be proved using similar methods as those employed in \cite{LeftDefiniteSL}. 
 Alternatively, consider the transformation  
 \begin{align*}
  \eta(t) = \frac{b+a}{2} + \frac{b-a}{2} \tanh(t), \quad t\in\R
 \end{align*}
 and note that the functions   
 \begin{align*}
   \psi(z,t) = \eta'(t)^{-\frac{1}{2}} \phi_a\left(z,\eta(t)\right), \quad t\in\R,~z\in\C
 \end{align*}
 are solutions of the left-definite Sturm--Liouville equation
 \begin{align}\label{eqnLDSP}
  -v''+  v = z\, v\, \mu 
 \end{align}
 on $\R$, where $\mu$ is the Borel measure on $\R$ which is given by
 \begin{align*}
  \mu([t_1,t_2)) = \int_{t_1}^{t_2} \eta'(t)  d\omega(\eta(t)), \quad t_1,\, t_2\in\R,~ t_1<t_2. 
 \end{align*}
 Considered in the Sobolev space $H^1(\R)$, the spectral problem \eqref{eqnLDSP} has the same spectrum as the operator $S$ with $\|\psi(\lambda,\cdot\,)\|_{H^1(\R)}=\gamma_\lambda$, $\lambda\in\sigma(S)$. 
 Hence an application of \cite[Theorem~7.5]{LeftDefiniteSL} yields the claim. 
 \end{proof}
 
 The growth of the mass distribution near the left endpoint is related to the growth of the spectral measure (cf.\ \cite[Theorem~7]{ko07}).
 In particular, it is possible to tell from the spectral measure whether $\omega$ is finite near $a$ or not.

\begin{corollary}\label{corISPreg}
 The measure $\omega$ is finite near $a$ if and only if the sum 
 \begin{align*}
  \sum_{\lambda\in\sigma(S)} \lambda^{-2} \gamma_\lambda^{-2} 
 \end{align*} 
 is finite.
\end{corollary}

\begin{proof}
 If the measure $\omega$ is finite near $a$, then we have \cite[Lemma~10.7]{measureSL} 
 \begin{align*}
  \int_a^b (b-x) \phi_a(\lambda,x)d\omega(x) = \frac{b-a}{\lambda}, \quad \lambda\in\sigma(S),
 \end{align*}
 which is square integrable by Parseval's identity. 
 Conversely, we may consider approximating measures $\omega_n$, $n\in\N$ as in the proof of Theorem~\ref{thmISPee}. 
 For each compactly supported continuous cut-off function $\chi$ which takes values in $[0,1]$ we have
 \begin{align*}
   \int_a^b (b-x)^2 \chi(x) d\omega_n(x) \leq \int_a^b (b-x)^2 d\omega_n(x) \leq (b-a)^2 \sum_{\lambda\in\sigma(S)} \lambda^{-2} \gamma_{\lambda}^{-2}, \quad n\in\N.
  \end{align*}
 Now since $\omega_n\rightharpoonup^\ast\omega$ as $n\rightarrow\infty$, we also infer that    
  \begin{align*}
   \int_a^b (b-x)^2 \chi(x) d\omega(x) \leq (b-a)^2 \sum_{\lambda\in\sigma(S)} \lambda^{-2} \gamma_\lambda^{-2}
  \end{align*}
  and hence $\omega$ is finite near $a$.
\end{proof}

Similarly, the measure $\omega$ is finite near the right endpoint $b$ if and only if the sum 
 \begin{align*}
  \sum_{\lambda\in\sigma(S)} \lambda^{-2} \dot{W}(\lambda)^{-2} \gamma_\lambda^2  
 \end{align*}
 is finite. More precisely, this follows from Corollary~\ref{corISPreg} upon replacing the roles of the endpoints and 
 \begin{align*}
  \dot{W}(\lambda)^2 = \gamma_\lambda^2 \, \|\phi_b(\lambda,\cdot\,)\|_{L^2((a,b);\omega)}^2, \quad \lambda\in\sigma(S),
 \end{align*}
 which holds in view of equation~\eqref{eqnWlam} in Lemma~\ref{lemWderlam}. 
 
 Finally, some kind of continuity for the inverse spectral problem may be drawn from these results. 
 Provided the norming constants converge pointwise, the corresponding measures in $\M_\sigma$ will converge in the weak$^\ast$ topology.

 \begin{corollary}\label{corInvCont}
  Suppose that $\sigma(S_n)\subseteq\sigma$ converges to some $\sigma_\infty\subseteq\sigma$ and that the norming constants  $\gamma_{n,\lambda}^2$ converge to some $\gamma_{\infty,\lambda}^2\in[0,\infty]$ for each $\lambda\in\sigma_\infty$.
  Then $\omega_n\rightharpoonup^\ast\omega$ for some $\omega\in\M_\sigma$ whose associated spectrum is $\sigma_\infty\backslash(\tau_a\cup\tau_b)$ and whose norming constants are 
  \begin{align*}
    \gamma_{\infty,\lambda}^2 \prod_{\kappa\in\tau_a} \biggl(1-\frac{\lambda}{\kappa}\biggr)^{-2}, \quad \lambda\in\sigma_\infty\backslash(\tau_a\cup\tau_b),
  \end{align*}
  where $\tau_a=\lbrace \lambda\in\sigma_\infty \,|\, \gamma_{\infty,\lambda}^2 = 0\rbrace$ and $\tau_b=\lbrace \lambda\in\sigma_\infty \,|\, \gamma_{\infty,\lambda}^2 = \infty\rbrace$.
 \end{corollary}

 \begin{proof}
  First of all suppose that $\omega_n\rightharpoonup^\ast\omega$ for some $\omega\in\M_\sigma$. 
  Then Proposition~\ref{propConvSM} shows that the spectrum and norming constants corresponding to $\omega$ are given as in the claim.  
  Now since this limit $\omega$ is uniquely determined by these quantities in view of Theorem~\ref{thmISPee}, the claim follows from compactness of $\M_\sigma$.  
 \end{proof}

\section{The inverse three-spectra problem}\label{secITS}

 The continuity results for the spectral quantities furthermore allow us to solve the inverse three-spectra problem for some arbitrary fixed interior point $c\in(a,b)$. 

\begin{theorem}\label{thmITSee}
 All sets $(\sigma,\sigma_a,\sigma_b)\in\T$ are the three spectra of a measure $\omega\in\M$, which is unique if and only if the sets $\sigma$, $\sigma_a$ and $\sigma_b$ are disjoint.
\end{theorem}

\begin{proof}
  First of all suppose that $\sigma$ is finite and that the intersection $\sigma\cap\sigma_a\cap\sigma_b$ is given by $\lbrace \lambda_1,\ldots,\lambda_k\rbrace$ for some $k\in\N_0$.
   Then for each $n\in\N$ and $d_1,\ldots,d_k\in\R^+$ we may consider the sets  $\sigma_{n,a}$ and $\sigma_{n,b}$, which are obtained from $\sigma_a$ and $\sigma_b$ upon replacing the values $\lambda_j$, $j=1,\ldots,k$ with $\lambda_j+d_j/n$ and $\lambda_j-1/n$ respectively. 
   For large enough $n\in\N$, the results in \cite{boypiv} show that there are measures $\omega_n$ with the three spectra $\sigma$, $\sigma_{n,a}$ and $\sigma_{n,b}$.  
   Moreover, using \eqref{eqngamma} one sees that for each $\lambda\in\sigma$ the limit of the corresponding norming constants $\gamma_{n,\lambda}^2$ as $n\rightarrow\infty$ is given by  
   \begin{align}\label{eqnTSnc1}
     -\frac{(b-a)(a-c)}{b-c} \frac{1}{\lambda} \prod_{\kappa\in\sigma\backslash\lbrace\lambda\rbrace} \biggl(1-\frac{\lambda}{\kappa}\biggr) \prod_{\mu_a\in\sigma_a}\biggl(1-\frac{\lambda}{\mu_a}\biggr) \prod_{\mu_b\in\sigma_b} \biggl(1-\frac{\lambda}{\mu_b}\biggr)^{-1}  
   \end{align}
   for all values $\lambda\not\in\sigma\cap\sigma_a\cap\sigma_b$ and by
   \begin{align}\label{eqnTSnc2}
     d_j \frac{(b-a)(a-c)}{b-c} \frac{1}{\lambda} \prod_{\kappa\in\sigma\backslash\lbrace\lambda\rbrace} \biggl(1-\frac{\lambda}{\kappa}\biggr)  \prod_{\mu_a\in\sigma_a\backslash\lbrace\lambda\rbrace} \biggl(1-\frac{\lambda}{\mu_a}\biggr) \prod_{\mu_b\in\sigma_b\backslash\lbrace\lambda\rbrace} \biggl(1-\frac{\lambda}{\mu_b}\biggr)^{-1} 
   \end{align}
   if $\lambda=\lambda_j$ for some $j\in\lbrace1,\ldots,k\rbrace$. 
   Now Corollary~\ref{corInvCont} and Proposition~\ref{propCont} show that $\omega_n\rightharpoonup^\ast\omega$ for some $\omega\in\M$ with the three spectra $\sigma$, $\sigma_a$, $\sigma_b$ and norming constants given by~\eqref{eqnTSnc1} and~\eqref{eqnTSnc2}. 
   In particular, note that we may prescribe the norming constants $\gamma_\lambda^2$, $\lambda\in\sigma\cap\sigma_a\cap\sigma_b$ without changing the three spectra. 
   
  Next suppose that $\sigma$ is infinite and fix some positive reals $\eta_\lambda^2\in\R^+$ for each $\lambda\in\sigma\cap\sigma_a\cap\sigma_b$. 
   From the considerations above, there are finite measures $\omega_n\in\M$, $n\in\N$ with the three spectra $\sigma\cap(0,n)$, $\sigma_{a}\cap(0,n)$, $\sigma_{b}\cap(0,n)$ and prescribed norming constants $\gamma_{n,\lambda}^2= \eta_\lambda^2$ for $\lambda\in\sigma\cap\sigma_a\cap\sigma_b\cap(0,n)$. 
   By construction, it follows from \eqref{eqngamma} that all norming constants $\gamma_{n,\lambda}^2$, $\lambda\in\sigma$ converge to a finite positive limit. 
   Now Corollary~\ref{corInvCont} shows that $\omega_n\rightharpoonup^\ast\omega$ for some $\omega\in\M$ with the three spectra $\sigma$, $\sigma_a$, $\sigma_b$ in view of Proposition~\ref{propCont} and norming constants $\gamma_\lambda^2=\eta_\lambda^2$ for $\lambda\in\sigma\cap\sigma_a\cap\sigma_b$. 
 
  Finally, if the sets $\sigma$, $\sigma_a$ and $\sigma_b$ are disjoint, then the spectral measure is given by \eqref{eqngamma}, which proves uniqueness in this case.
\end{proof}

 Note that the proof of Theorem~\ref{thmITSee} also contains a description of all measures in $\M$ with the same three spectra $\sigma$, $\sigma_a$ and $\sigma_b$ in terms of the corresponding spectral measures. 
  More precisely, the measure $\omega$ has the three spectra $\sigma$, $\sigma_a$ and $\sigma_b$ if and only if $\supp(\rho)=\sigma$ and 
   \begin{align*}
    \gamma_\lambda^2 = -\frac{(b-a)(a-c)}{b-c} \frac{1}{\lambda}\prod_{\kappa\in\sigma\backslash\lbrace\lambda\rbrace} \biggl(1-\frac{\lambda}{\kappa}\biggr)  \prod_{\mu_a\in\sigma_a}\biggl(1-\frac{\lambda}{\mu_a}\biggr) \prod_{\mu_b\in\sigma_b} \biggl(1-\frac{\lambda}{\mu_b}\biggr)^{-1}
   \end{align*} 
   for all $\lambda\in\sigma\backslash(\sigma_a\cap\sigma_b)$. 
 This means that only this part of the norming constants is uniquely determined by the three spectra, whereas the remaining norming constants $\gamma_\lambda^2$, $\lambda\in\sigma\cap\sigma_a\cap\sigma_b$ may be chosen arbitrarily in $\R^+$. 
 Also note that additional knowledge of these remaining norming constants uniquely determines the spectral measure and hence also the weight measure in $\M$. 
 Since the norming constants and coupling constants are simply related by~\eqref{eqnWlam}, we may summarize these considerations in the following result. 

 \begin{corollary}\label{corTSIP}
  Given discrete sets $(\sigma,\sigma_a,\sigma_b)\in\T$ and numbers $c_\lambda\in\R^+$ for each $\lambda\in\sigma\cap\sigma_a\cap\sigma_b$, there is a unique $\omega\in\M$ whose associated three spectra are $\sigma$, $\sigma_a$ and $\sigma_b$ and whose coupling constants for $\lambda\in\sigma\cap\sigma_a\cap\sigma_b$ are $c_\lambda$. 
 \end{corollary}

 As in the preceding section we may characterize those endpoints near which the measure $\omega$ is finite in terms of the spectral data in Corollary~\ref{corTSIP}.  
 In fact, Corollary~\ref{corISPreg} and equation~\eqref{eqnWlam} show that $\omega$ is finite near $a$ if and only if the sum
   \begin{align*}
     \sum_{\lambda\in\sigma(S)} \lambda^{-2} |\dot{W}(\lambda)|^{-1} c_\lambda
   \end{align*}
 is finite and similarly that $\omega$ is finite near $b$ if and only if the sum
    \begin{align*}
     \sum_{\lambda\in\sigma(S)} \lambda^{-2} |\dot{W}(\lambda)|^{-1} c_\lambda^{-1}
   \end{align*}
 is finite. Hereby note that $W$ and all coupling constants corresponding to eigenvalues $\lambda\in\sigma(S)$ are explicitly given in terms of the spectral data in Corollary~\ref{corTSIP}.

 Finally, let us mention that using compactness of $\M_\sigma$, inverse uniqueness in Corollary~\ref{corTSIP} and the continuity results in Section~\ref{secCont}, it is again possible to deduce some kind of continuity 
 for the inverse problem on $\M_\sigma$. This can be done in much the same manner as in Corollary~\ref{corInvCont}.


\begin{thebibliography}{XX}

\bibitem{bss}
R.\ Beals, D.\ H.\ Sattinger and J.\ Szmigielski, {\em Multipeakons and the classical moment problem}, Adv.\ Math.\ {\bf 154} (2000), no.~2, 229--257.

\bibitem{boypiv}
O.\ Boyko and V.\ Pivovarchik, {\em The inverse three-spectra problem for a Stieltjes string and the inverse problem with one-dimensional damping}, Inverse Problems {\bf 24} (2008), no.~1, 015019, 13 pp.

\bibitem{dymmck}
H.\ Dym and H.\ P.\ McKean, {\em Gaussian processes, function theory and the inverse spectral problem}, Probability and Mathematical Statistics, Vol.\ 31, Academic Press, New York-London, 1976.

\bibitem{LeftDefiniteSL}
J.\ Eckhardt, {\em Direct and inverse spectral theory of singular left-definite Sturm--Liouville operators}, J.\ Differential Equations {\bf 253} (2012), no.~2, 604--634. 

\bibitem{measureSL}
J.\ Eckhardt and G.\ Teschl, {\em Sturm--Liouville operators with measure-valued coefficients}, J.\ Anal.\ Math.\ {\bf 120} (2013), no.~1, 151--224.

\bibitem{tsub}
N.\ Falkner and G.\ Teschl, {\em On the substitution rule for Lebesgue--Stieltjes  integrals}, Expo.\ Math.\ {\bf 30} (2012), no.~4, 412--418.

\bibitem{gessimts}
F.\ Gesztesy and B.\ Simon, {\em On the determination of a potential from three spectra}, Differential operators and spectral theory, 85--92, Amer.\ Math.\ Soc.\ Transl.\ Ser.\ 2, 189, AMS, Providence, RI, 1999.

\bibitem{gokr}
I.\ C.\ Gohberg and M.\ G.\ Kre\u{\i}n, {Introduction to the theory of linear nonselfadjoint operators}, Transl.\ of Math.\ Mon.\ {\bf 18}, Amer.\ Math.\ Soc., Providence, RI, 1969.

\bibitem{hrymyk}
R.\ O.\ Hryniv and Ya.\ V.\ Mykytyuk, {\em Inverse spectral problems for Sturm--Liouville operators with singular potentials. Part III: Reconstruction by three spectra}, J.\ Math.\ Anal.\ Appl.\ {\bf 284} (2003), no.~2, 626--646.

\bibitem{ka67}
I.\ S.\ Kac, {\em The existence of spectral functions of generalized second order differential systems with boundary conditions at the singular end}, Amer.\ Math.\ Soc.\ Transl.\ (2) {\bf 62} (1967), 204--262.

\bibitem{kackrein}
I.\ S.\ Kac and M.\ G.\ Krein, {\em On the spectral functions of the string}, Amer.\ Math.\ Soc.\ Transl.\ Ser.\ 2, 103, AMS, Providence, RI, 1974.

\bibitem{kac}
I.\ S.\ Kats, {\em The spectral theory of a string}, Ukrainian Math.\ J.\ {\bf 46} (1994), no.~3, 159--182.

\bibitem{koe}
H.\ K\"onig, {\em Eigenvalue distribution of compact operators}, Operator Theory: Advances and Applications, 16, Birkh\"auser Verlag, Basel, 1986.

\bibitem{ko75}
S.\ Kotani, {\em On a generalized Sturm--Liouville operator with a singular boundary}, J.\ Math.\ Kyoto Univ.\ {\bf 15} (1975), no.~2, 423--454.

\bibitem{ko76}
S.\ Kotani, {\em A remark to the ordering theorem of L.\ de Branges}, J.\ Math.\ Kyoto Univ.\ {\bf 16} (1976), no.~3, 665--674.

\bibitem{ko07}
S.\ Kotani, {\em Krein's strings with singular left boundary}, Rep.\ Math.\ Phys.\ {\bf 59} (2007), no.~3, 305--316.

\bibitem{kotwat}
S.\ Kotani and S.\ Watanabe, {\em Kre\u{\i}n's spectral theory of strings and generalized diffusion processes}, Lecture Notes in Math., 923, Springer, Berlin-New York, 1982.

\bibitem{pivovarchik}
V.\ Pivovarchik, {\em An inverse Sturm--Liouville problem by three spectra}, Integral Equations Operator Theory {\bf 34} (1999), no.~2, 234---243.

\bibitem{tschroe}
G.\ Teschl, {\em Mathematical Methods in Quantum Mechanics; With Applications to Schr\"odinger Operators},
Graduate Studies in Mathematics {\bf 99}, Amer.\ Math.\ Soc., Providence, RI, 2009.

\end{thebibliography}
\end{document}